\documentclass[12pt]{amsart}

\usepackage[T1]{fontenc}
\usepackage[american]{babel}
\usepackage{amssymb}
\usepackage{amsmath}
\usepackage{amscd}
\usepackage{graphicx}
\usepackage{xcolor}
\usepackage{booktabs}
\usepackage{mathtools}
\usepackage{dsfont}
\usepackage{float}
\usepackage{tikz}
\usepackage{etoolbox}
\usepackage{enumerate}
\usepackage{caption}
\usepackage{subcaption}
\usepackage[textheight=611pt,textwidth=456pt,centering]{geometry}

\usepackage[breaklinks,colorlinks,citecolor=blue,linkcolor=blue,
  pdftitle={Irreducible infeasible subsystems of semidefinite systems}]{hyperref}

\graphicspath{{pics/}}

\numberwithin{equation}{section}
\newtheorem{theorem}{Theorem}
\newtheorem{proposition}[theorem]{Proposition}
\newtheorem{lemma}[theorem]{Lemma}
\newtheorem{corollary}[theorem]{Corollary}

\theoremstyle{definition}
\newtheorem{example}[theorem]{Example}
\newtheorem{remark1}[theorem]{Remark}
\newtheorem{assumption1}[theorem]{Assumption}
\newtheorem{openproblem1}[theorem]{Open problem}
\newtheorem{definition}[theorem]{Definition}

\newenvironment{remark}{\begin{remark1}\rm}{\end{remark1}}
\newenvironment{assumption}{\begin{assumption1}\rm}{\end{assumption1}}

\numberwithin{theorem}{section}

\newcounter{FNC}[page]
\def\newfootnote#1{{\addtocounter{FNC}{2}$^\fnsymbol{FNC}$%
     \let\thefootnote\relax\footnotetext{$^\fnsymbol{FNC}$#1}}}

\newcommand{\R}{\mathds{R}}

\newcommand{\sym}{\mathcal{S}}

\newcommand{\BS}{\mathrm{BS}}
\newcommand{\card}[1]{\lvert{#1}\rvert}
\newcommand{\norm}[1]{\lVert{#1}\rVert}
\newcommand{\define}{\coloneqq}
\newcommand{\suchthat}{\,:\,}
\newcommand{\sprod}[2]{{#1} \bullet {#2}}
\newcommand{\T}{^\top}
\newcommand{\id}{\mathds{I}} 
\newcommand{\floor}[1]{\lfloor{#1}\rfloor}

\DeclareMathOperator{\conv}{conv}

\DeclareMathOperator{\diag}{diag}

\DeclareMathOperator{\tr}{tr}

\title[Irreducible infeasible subsystems of semidefinite systems]{Irreducible
infeasible subsystems of\\ semidefinite systems}

\author{Kai Kellner}
\address[Kai Kellner]{Frankfurt am Main, Germany}
\email{kellner.kai@gmx-topmail.de}

\author{Marc E. Pfetsch}
\address[Marc E. Pfetsch]{Department of Mathematics, TU Darmstadt, Dolivostr.\
15, 64293 Darmstadt, Germany}
\email{pfetsch@mathematik.tu-darmstadt.de}

\author{Thorsten Theobald}
\address[Thorsten Theobald]{Goethe-Universit\"at, FB 12 -- Institut f\"ur
Mathematik, Postfach 11 19 32, 60054 Frankfurt am Main, Germany}
\email{theobald@math.uni-frankfurt.de}

\date{\today}

\begin{document}

\begin{abstract}
  Farkas' lemma for semidefinite programming characterizes semidefinite
  feasibility of linear matrix pencils in terms of an alternative
  spectrahedron.  In the well-studied special case of linear programming, a
  theorem by Gleeson and Ryan states that the index sets of irreducible
  infeasible subsystems are exactly the supports of the vertices of the 
  corresponding alternative polyhedron.

  We show that one direction of this theorem can be generalized to the
  nonlinear situation of extreme points of general spectrahedra.  The
  reverse direction, however, is not true in general, which we show by
  means of counterexamples. On the positive side, an irreducible infeasible
  block subsystem is obtained whenever the extreme point has minimal block
  support. Motivated by results from sparse recovery, we provide a
  criterion for the uniqueness of solutions of semidefinite block systems.
\end{abstract}

\maketitle

\section{Introduction}

The structure of infeasible linear inequality systems is quite well
understood. In particular, Farkas' Lemma, also called Theorem of the
Alternative, gives a characterization of infeasibility (see, e.g.,
\cite{schrijver-book}). Moreover, the basic building blocks are so-called
Irreducible Infeasible Systems (IISs, also called Irreducible Inconsistent
Systems), i.e., infeasible subsystems such
that every proper subsystem is feasible. An extension of the Theorem of the
Alternative, due to Gleeson and Ryan~\cite{GleR90}, states that the IISs of
an infeasible linear inequality system correspond exactly to the vertices
of a so-called alternative polyhedron (see Theorem~\ref{thm:gleeson-ryan}). These IISs provide a means to
analyze infeasibilities of a system, see, e.g.,
\cite{Chi97b,ChiD91,VanLoon1981} 
and the book~\cite{Chi2008}.  Today,
standard optimization software can compute (hopefully) small IISs. Further
investigations include the mixed-integer case~\cite{GuiC99} and the
application within Benders' decomposition~\cite{CodF04}.

In this article, we consider infeasible systems in semidefinite form.
There are well-known generalizations of the Theorem of the
Alternative to this setting (see, e.g., \cite{tuncel2010}), 
although one has to be more careful, since
feasibility might only be attained in the limit -- see 
Proposition~\ref{prop:SDPAlternative} for a more
precise statement. As in the linear case, solutions of certain alternative
systems give a certificate of the (weak) infeasibility of a semidefinite
system.

In this context, the following natural questions arise: How can infeasible
semidefinite systems be analyzed? What can be said about
the structure of irreducible infeasible semidefinite systems? Moreover, is
there a generalization of the theorem of Gleeson and Ryan to this setting?

These questions are motivated by solving mixed-integer semidefinite
programs using branch-and-bound in which an SDP is solved in every node
(see, e.g., \cite{gpu-2017}). Then, it often happens that these SDPs turn
out to be infeasible. One would now like to learn from this infeasibility
in order to strengthen the relaxations of other nodes. This is done in
mixed-integer and SAT solvers, see, e.g., \cite{Ach07b,WitBH2017}.

To come up with an appropriate definition of an IIS for a semidefinite
system it appears to be very natural to consider block systems. Then, an IIS is
given by an inclusion minimal set of infeasible block subsystems (see
Definition~\ref{def:BlockSubsystems}). We will show in
Section~\ref{sec:AlternativeSystems} that one direction of the above
mentioned connection can be generalized: there always exists an extreme
point of the alternative system that corresponds to a given IIS, see
Theorem~\ref{th:iis-gives-extreme}.  The reverse direction, however, is not
true in general, which we show and discuss via two counterexamples, see
Examples~\ref{ex:blocklinear} and~\ref{ex:Blocksdp}. On the positive side,
whenever an extreme point has (inclusionwise) minimal block support, the
corresponding subsystem forms indeed an IIS. This leads to the general task
to compute such points.

In the particular case in which the alternative semidefinite system has a
unique solution, this algorithmic challenge simplifies to solving one
semidefinite program. Motivated by results from sparse recovery, we provide
a criterion for the uniqueness of solutions of semidefinite block
systems. In Section~\ref{se:recovery}, we generalize the results in
\cite{elhamifar-vidal-2012,kdxh-2011,wang-tang-2009,wxt-2011} to give
unique recovery characterizations for a block semidefinite system in
Theorem~\ref{thm:blocksparse1}. Further perspectives
and open questions are given in Section~\ref{se:perspectives}.

\section{Infeasible Systems and Block Structure}
\label{sec:BlockStructure}

We use the following notation. Let $\sym^n$ be the set of
all (real) symmetric $n \times n$ matrices, and ``$\succeq$'' denotes
that a symmetric matrix is positive semidefinite (psd).
For a matrix $A \in \sym^n$ and
$I \subseteq [n] \define \{1, \dots, n\}$, let $A_I$ be the submatrix
containing the rows and columns of~$A$ indexed by~$I$. 
For $A$, $B \in \sym^n$, we denote the inner
product by
\[
\sprod{A}{B} = \tr(A\T B)
= \sum_{i,j=1}^n A_{ij}\, B_{ij},
\]
where $\tr(\cdot)$ denotes the trace.  $\|A\|_2$ denotes the
operator norm $\|A\|_2 = \max_{1 \le j \le n} |\lambda_j(A)|$, where
$\lambda_1(A), \ldots, \lambda_n(A)$ are the eigenvalues of $A$.

Throughout the paper, let 
$A_0, \dots, A_m \in \sym^n$. For $y \in \R^m$, we consider the
linear \emph{(matrix) pencil}
\[
A(y) \define A_0 - \sum_{i=1}^m y_i A_i
\]
and the \emph{linear matrix inequality} (LMI) $A(y) \succeq 0$. With
respect to infeasibility, we will use the following result, where $\id$ denotes 
the identity matrix.

\begin{proposition}\label{prop:SDPAlternative}
  Either $A(y) + \varepsilon \id \succeq 0$ is feasible for every
  $\varepsilon > 0$ or there exists $X \succeq 0$ with $\sprod{A_i}{X} =
  0$, $i \in [m]$, and $\sprod{A_0}{X} = -1$.
\end{proposition}

This statement is equivalent to Sturm's Farkas' Lemma for semidefinite
programming
(see Lemmas~3.1.1 and 3.1.2 in \cite{klep-schweighofer-2013}) and
a variation
of Theorem~2.21 in \cite{tuncel2010}, and its proof is provided for completeness.

\begin{proof}
  Consider the following dual pair of semidefinite programs (SDPs):
  \begin{align}
    & \inf\; \{ \eta \suchthat A(y) + \eta\, \id \succeq 0,\; \eta \geq 
0\},\label{eq:AltD} \\
    & \sup\; \{ - \sprod{A_0}{X} \suchthat \sprod{A_i}{X} = 0,\; i \in [m],\;
    \tr(X) \leq 1,\; X \succeq 0\}.\label{eq:AltP}
  \end{align}
  Setting $y = 0$, $\eta = \norm{A_0}_2 + 1$ shows that~\eqref{eq:AltD} has
  a Slater point. Moreover, $X = 0$ is feasible for~\eqref{eq:AltP}. The
  strong duality theorem (see, e.g., Theorem~2.14 in \cite{tuncel2010}) implies
  that~\eqref{eq:AltP} attains its optimal value and the objective values
  are the same.

  Suppose that no $X \succeq 0$ with $\sprod{A_i}{X} = 0$, $i \in [m]$,
  $\sprod{A_0}{X} = -1$ exists. By scaling, this implies that no such $X$
  exists with $\sprod{A_0}{X} < 0$. 
  And since the zero matrix is feasible 
  for~\eqref{eq:AltP}, the optimal value of \eqref{eq:AltP} is 0.
  By the strong duality theorem, the optimal value of~\eqref{eq:AltD} is
  also 0.
  Either~\eqref{eq:AltD} attains this value and we are done, or
  there exists a sequence $(y^k, \eta_k)$ such that $A(y^k) + \eta_k\, \id
  \succeq 0$ and $\eta_k \searrow 0$. This implies the theorem. 
\end{proof}

\begin{remark}
  In slight deviation from parts of the literature, we call $A(y) \succeq
  0$ \emph{weakly feasible}, if for every $\varepsilon > 0$ the system
  $A(y) + \varepsilon \id \succeq 0$ is feasible; compare this, for instance,
  to the definition in~\cite{tuncel2010}, which requires $\tilde{A}_0 -
  \sum_{i=1}^m y_i\, A_i \succeq 0$ to be feasible for some $\tilde{A}_0$
  such that $\norm{A_0 - \tilde{A}_0}_2 < \varepsilon$. Moreover, $A(y)
  \succeq 0$ is \emph{weakly infeasible} if it is not weakly feasible. Note
  the slight inaccuracy of this naming convention, which should, however,
  not lead to confusion in the present paper.
\end{remark}

\begin{corollary}
  Assume that there exists $\bar{X} \succ 0$ with $\sprod{A_i}{\bar{X}}
  = 0$, $i \in [m]$. Then, either $A(y) \succeq 0$ is feasible or there
  exists $X \succeq 0$ with $\sprod{A_i}{X} = 0$, $i \in [m]$, and
  $\sprod{A_0}{X} = -1$.
\end{corollary}

\begin{proof}
  By scaling~$\bar{X}$ to satisfy $\tr(\bar{X}) \leq 1$, the assumption
  guarantees that~\eqref{eq:AltP} above has a Slater point and therefore
  that the optimal value of~\eqref{eq:AltD} is attained, see, e.g.,
  Corollary 2.17 in \cite{tuncel2010}. The remaining part of the proof is as
  for the one of Proposition~\ref{prop:SDPAlternative}.
\end{proof}

Our subsequent definition of an alternative spectrahedron will
allow to handle structured semidefinite systems. To motivate this
viewpoint, consider a simple example where the goal is to check
whether two given halfplanes 
$H_i = \{y \in \R^2 \suchthat \alpha_i y_1 + \beta_i y_2 + \gamma_i \geq 0\}$
($i \in [2]$) and a given disc $D = \{y \in \R^2 \suchthat \norm{y-c}_2^2 \le r^2\}$
in the Euclidean plane $\R^2$ with center $c \in \R^2$ have a common point.
The smallest LMI-representation (w.r.t.\ matrix size) of $D$ is
\[
K(r,c;y) \define \begin{pmatrix} r + c_1 - y_1 & y_2-c_2 \\ y_2 - c_2 & r - c_1 + y_1
\end{pmatrix} \succeq 0,
\]
and thus the existence of a point in $H_1 \cap H_2 \cap D$ is equivalent
to the feasibility of the LMI
\begin{align}
  \label{eq:structured1}
  A(y) = 
  \begin{pmatrix} 
    \alpha_1 y_1 + \beta_1 y_2 + \gamma_1 & & \\ 
    & \alpha_2 y_1 + \beta_2 y_2 + \gamma_2 & \\ & & K(r,c;y)
  \end{pmatrix} \succeq 0.
\end{align}
In order to capture such natural structure within semidefinite systems,
one arrives at block systems. In particular, already in the simple
example this allows then to consider the $2 \times 2$-subsystem of the
disc as an entity. Formally, this yields the following.

\begin{definition}
  Let $k \ge 1$ and $B_1, \dots, B_k \neq \emptyset$ a partition of the
  set $[n]$. 
  A linear pencil $A(y)$ is in \emph{block-diagonal form} with blocks
  $B_1, \ldots, B_k$ if and only if each $A_i$ is 0 outside of the blocks $B_1, \dots, B_k$,
  i.e., $(A_i)_{st} = 0$ for all $(s,t) \notin (B_1 \times B_1) \cup \dots
  \cup (B_k \times B_k)$ and all $i \in [m]$. Note that the blocks might be
  \emph{decomposable}, i.e., at least one block consists of blocks of
  smaller size while still retaining the block structure of $A(y)$.
\end{definition}

\begin{assumption}\label{aspt:nontrivial}
  To avoid trivial infeasibilities, we will assume that for each block
  $B_i$, $i \in [k]$, there exists $y \in \R^m$ such that $A(y)_{B_i}
  \succeq 0$ is weakly feasible.
\end{assumption}

\begin{definition}\label{def:BlockSubsystems}
  Let $A(y)$ be in block-diagonal form with blocks $B_1, \ldots, 
B_k$.
  \begin{enumerate}[(a)]
  \item For $I \subseteq [k]$, the \emph{block subsystem of $A(y)$ with respect 
to $I$}
     is given by $A(y)_{B(I)}$ for the index set $B(I) \define
    \bigcup_{i\in I} B_i$. By convention, $B(\emptyset) = \emptyset$ and
    $A(y)_{\emptyset}$ is a feasible system.
  \item A block subsystem with respect to some $I\subseteq [k]$ is an 
\emph{irreducible
      infeasible subsystem (IIS)} iff $A(y)_{B(I)}\succeq 0$ is weakly
    infeasible, but $A(y)_{B(I')} \succeq 0$ is weakly feasible for
    all $I'\subsetneq I$.
  \item Given a matrix $X \in \sym^n$, its \emph{block support} $\BS(X)$ is 
defined as
    \[
    \BS(X) \define \{i \in [k] \suchthat X_{B_i} \neq 0\}.
    \]
  \end{enumerate}
\end{definition}

\begin{remark}
  Linear inequality systems arise if all matrices $A_0, \dots, A_m$ of $A(y)$ 
are
  diagonal. In this case, each inequality is of the form
  \[
  (A_0)_{jj} - \sum_{i=1}^m y_i\, (A_i)_{jj} \geq 0,\; j \in [n].
  \]
  If this system is written as $D x \leq d$, then IISs correspond to
  infeasible subsystems of $Dx \leq d$ such that each proper subsystem is
  feasible.

  The linear case arises, in particular, if the block system satisfies $k =
  n$ (and hence $\card{B_i} = 1$); then the blocks are not decomposable. However,
  it is also possible that the blocks are decomposable. In this
  case, the system consists of $k$ linear inequality systems $D_1 x
  \leq d_1$, \dots, $D_k x \leq d_k$, each defining a polyhedron. If the
  intersection of these polyhedra is empty, the original LMI is infeasible;
  see Example~\ref{ex:blocklinear} below.
\end{remark}

\begin{remark}
  An alternative way to define IISs would be to consider subsets $S
  \subseteq [n]$ such that $A(y)_S \succeq 0$ is (weakly) infeasible, but
  $A(y)_{\hat{S}} \succeq 0$ is (weakly) feasible for every proper subset
  $\hat{S}$ of~$S$. However, this definition would not retain the
  structure within semidefinite systems such as~\eqref{eq:structured1}. 
  As a consequence, we currently do not know to which extent our subsequent
  investigations can be transferred to that model.
\end{remark}

\section{Alternative Systems}
\label{sec:AlternativeSystems}

In view of Proposition~\ref{prop:SDPAlternative}, we define the following,
where the abbreviation $\Sigma$ for the LMI $A(y) \succeq 0$
will allow for a convenient notation. For general background on spectrahedra,
we refer to \cite{bpt-2013,theobald-spectrahedral-2017}.

\begin{definition}
  The \emph{alternative spectrahedron} for $\Sigma: A(y) \succeq 0$ is
  \[
  S(\Sigma) \define \{X \succeq 0 \suchthat \sprod{A_i}{X} = 0,\; i \in
  [m],\; \sprod{A_0}{X} = -1\}.
  \]
\end{definition}

\begin{assumption}\label{aspt:block}
  By standard polarity theory, 
  a block structure of the system can also be assumed for $X
  \in S(\Sigma)$.  Thus, we only consider matrices $X \in S(\Sigma)$ in
  block-diagonal form, where the blocks are indexed by $\BS(X)$.
\end{assumption}

The definition of the alternative spectrahedron immediately implies:

\begin{lemma}\label{lemma:IIS_MinBlockSupport}
  Let $\Sigma: A(y) \succeq 0$ be a weakly infeasible semidefinite system
  with blocks $B_1, \ldots, B_k$.
  \begin{enumerate}[(a)]
  \item For any $X \in S(\Sigma)$, there exists an infeasible
    subsystem of $\Sigma$ with block support contained in $\BS(X)$.
  \item For any $X \in S(\Sigma)$ with inclusion-minimal block support,
    the index set $\BS(X)$ defines an IIS of $\Sigma$.
  \end{enumerate}
\end{lemma}

As mentioned in the introduction, in the linear case there exists a
characterization of IISs:

\begin{theorem}[Gleeson and Ryan~\cite{GleR90}]\label{thm:gleeson-ryan}
  Consider an infeasible system $\Sigma: A x \leq b$, where $A \in \R^{m
    \times n}$, $b \in \R^m$. The index sets of the IISs of $\Sigma$ are
  exactly the support sets of the vertices of the alternative polyhedron
  \[
  P(\Sigma) = \{y \in \R^m \suchthat y\T A = 0,\; y\T b =-1,\; y \geq 0\}.
  \]
\end{theorem}

A proof can be found in~\cite{GleR90} and~\cite{pfetsch2003}. Note that in
the non decomposable linear case, $P(\Sigma)$ is equivalent to the alternative
spectrahedron $S(\Sigma)$.

One goal of this paper is to investigate whether/how far Theorem~\ref{thm:gleeson-ryan}
generalizes to the spectrahedral situation. We can show that one of the
directions can be generalized.

\begin{theorem}\label{th:iis-gives-extreme}
  Let $\Sigma: A(y) \succeq 0$ be a weakly infeasible LMI with blocks $B_1, \ldots, B_k$.
  For each index
  set $I$ of an IIS, there exists an extremal point of $S(\Sigma)$ with
  block support $I$.
\end{theorem}

The following proof proceeds by revealing the convex-geometric structure of
the alternative spectrahedron.

\begin{proof}
  Without any loss of generality, we can assume that $I = \{1, \dots, t\}$ for
  some $t \in [k]$. By Proposition~\ref{prop:SDPAlternative}, the alternative
  spectrahedron $S(\Sigma)$ contains a feasible point $X$ supported exactly
  on the blocks $B_1,\ldots,B_t$.  In order to show that the alternative
  spectrahedron contains an extremal point with block support $\{1, \ldots,
  t\}$, we first observe that $S(\Sigma)$ has at least one extremal point.
  This follows from the fact that the positive semidefinite cone is pointed
  and thus any slice of a subspace with this cone cannot have a nontrivial
  lineality space either.

  By Theorem 18.5 in~\cite{rockafellar-book}, 
  the alternative spectrahedron can be written in the form
  \[
  S(\Sigma) = \conv(E \cup F),
  \]
  where $E$ is the set of its extremal points and $F$ is the set of
  extremal directions of $S(\Sigma)$.  Hence, by a general version of
  Carath\'{e}odory's Theorem (see Theorem 17.1 in \cite{rockafellar-book}),
  there exist $r \ge 1$, $s \ge 0$, extremal points $V^{(1)}, \ldots,
  V^{(r)}$ and extremal rays $W^{(1)}, \ldots, W^{(s)}$ of the alternative
  spectrahedron such that
  \[
  X \ = \ \sum_{i=1}^r \lambda_i\, V^{(i)} + \sum_{j=1}^s \mu_j\, W^{(j)}
  \]
  with $\lambda_i$, $\mu_j > 0$ and $\sum_{i=1}^r \lambda_i = 1$. Since
  $V^{(i)}, W^{(j)}$ are positive semidefinite and $\lambda_i$, $\mu_j >
  0$, the block support of each $V^{(i)}$, $W^{(j)}$ must be contained in
  the block support of $X$. Due to the minimality of $I$, all $V^{(i)}$
  must have the same block support. Hence, the block support of $V^{(1)}$
  is exactly $I$, so that it is an extremal point with the desired property.
\end{proof}

We also provide the following shorter proof, which, however, reveals less 
structural insights.

\begin{proof}[Alternative proof.]
  Consider the intersection
  \[
  S' \define S(\Sigma) \cap \{X \suchthat X_{B_i} = 0, \; i \notin I\}.
  \]
  Then, $S'$ has an extreme point~$X' \in S'$, since it is the intersection
  of the pointed positive semidefinite cone with an affine space and
  therefore also pointed. Now let $I' \define \BS(X')$ be the block support
  of~$X'$. Then, $I' \subseteq I$ by construction. If $I' = I$, then we are done,
  since $X'$ is an extreme point of~$S(\Sigma)$ as well: Assume $X' = \lambda Z + (1-
  \lambda) Y$, $0 < \lambda < 1$, would be the strict convex combination of
  two other feasible points~$Y$ and $Z$, such that w.l.o.g.\ $Z$ has a
  support in a block~$B$ outside of $I'$. Then,
  \[
  \underbrace{\tr(X'_B)}_{= 0} = \lambda \underbrace{\tr(Z_B)}_{>0} + (1 -
  \lambda) \underbrace{\tr(Y_B)}_{\geq 0},
  \]
  would give a contradiction.

  Moreover, if $I' \subsetneq I$, then $X'$ shows that $A(y)_{B(I')} \succeq 0$ is
  infeasible. Thus, $I$ would not be minimal.
\end{proof}

The converse of this theorem is, however, not true in general. 
This direction may already fail in the presence of blocks of 
size~2. We will demonstrate
this by two counterexamples. The first one is linear, but decomposable. 
The second one is not decomposable, but nonlinear.

\begin{figure}
  \begin{center}
  \includegraphics[width=0.3\textwidth]{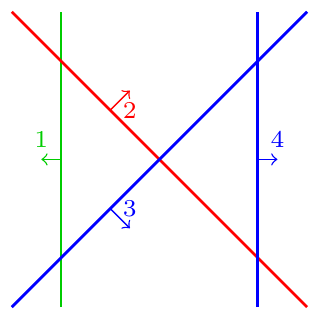}
  \end{center}
  \caption{Illustration for Example~\ref{ex:blocklinear}.}
  \label{fig:blocklinear}
\end{figure}

\begin{example}\label{ex:blocklinear}
  Let $m = 2$, $k = 3$ and
  \begin{align*}
    A_0 =
    \begin{pmatrix}
      0 &    &    &   \\
        & -1 &    &   \\
        &    & -1 & 0 \\
        &    &  0 & -2
    \end{pmatrix},
    \quad
    A_1 =
    \begin{pmatrix}
      1 &    &    &    \\
        & -1 &    &    \\
        &    & -1 &  0 \\
        &    &  0 &  -1
    \end{pmatrix},
    \quad
    A_2 =
    \begin{pmatrix}
      0 &    &   &   \\
        & -1 &   &   \\
        &    & 1 & 0 \\
        &    & 0 & 0
    \end{pmatrix}.
  \end{align*}
  The blocks are $B_1 = \{1\}$, $B_2 = \{2\}$, $B_3 = \{3, 4\}$, and this
  example corresponds to the three polyhedra
  \begin{align*}
    & P_1 \define \{y \in \R^2 \suchthat y_1 \leq 0\},\quad
    P_2 \define \{y \in \R^2 \suchthat y_1 + y_2 \geq 1\},\\
    & P_3 \define \{y \in \R^2 \suchthat -y_1 + y_2 \leq -1,\; y_1 \geq 2\},
  \end{align*}
  see Figure~\ref{fig:blocklinear} for an illustration. In this case, only
  the diagonal elements of the points $X$ in the alternative
  spectrahedron are relevant, which can be formulated as the polyhedron
  \begin{align*}
    S(\Sigma) \define \left\{ x \in \R^{1 + 1 + 2} \suchthat
      \begin{pmatrix}
        0 & -1 & -1 & -2 \\
        1 & -1 & -1 & -1 \\
        0 & -1 & \phantom{-} 1 & \phantom{-}0
      \end{pmatrix}
      x = \begin{pmatrix} -1 \\ \phantom{-}0 \\ \phantom{-}0 \end{pmatrix}
      ,\; x \geq 0\right\}.
  \end{align*}
  $S(\Sigma)$ is a one-dimensional polytope with the two vertices
  \[
  (1, \tfrac{1}{2}, \tfrac{1}{2}, 0)^\top \quad\text{and}\quad
  (\tfrac{1}{2}, 0, 0, \tfrac{1}{2})^\top.
  \]

  For the vertex $\tilde{x} = (1, \tfrac{1}{2}, \tfrac{1}{2}, 0)\T$
  of~$S(\Sigma)$, we have $\BS(\tilde{x}) = \{1,2,3\}$. However, this does not
  correspond to an IIS, since $\{1,3\}$ gives a proper subsystem that is
  infeasible.
\end{example}

To come up with non-decomposable blocks, the next counterexample deals with
a deformed version.

\begin{example}\label{ex:Blocksdp}
  For $\varepsilon \geq 0$, consider the linear matrix pencil given by
  \begin{align*}
    A_0 = 
    \begin{pmatrix} 
      0 &&& \\ 
      & -1 && \\ 
      && -1 & \varepsilon \\ 
      && \varepsilon & -2 
    \end{pmatrix}
  \end{align*}
  and the matrices $A_1$ and $A_2$ of Example~\ref{ex:blocklinear}. For
  $\varepsilon=0$, the system $\Sigma: A(y) \succeq 0$ specializes to
  Example~\ref{ex:blocklinear}. For $\varepsilon>0$ the two lines in
  Figure~\ref{fig:blocklinear} indexed by 3 and 4
  deform to a quadratic curve; see
  Figure~\ref{fig:deformed}. Note that the quadratic curve has a second
  component corresponding to the lower right block being negative definite.

  \begin{figure}
    \begin{center}
    \includegraphics[width=0.4\textwidth]{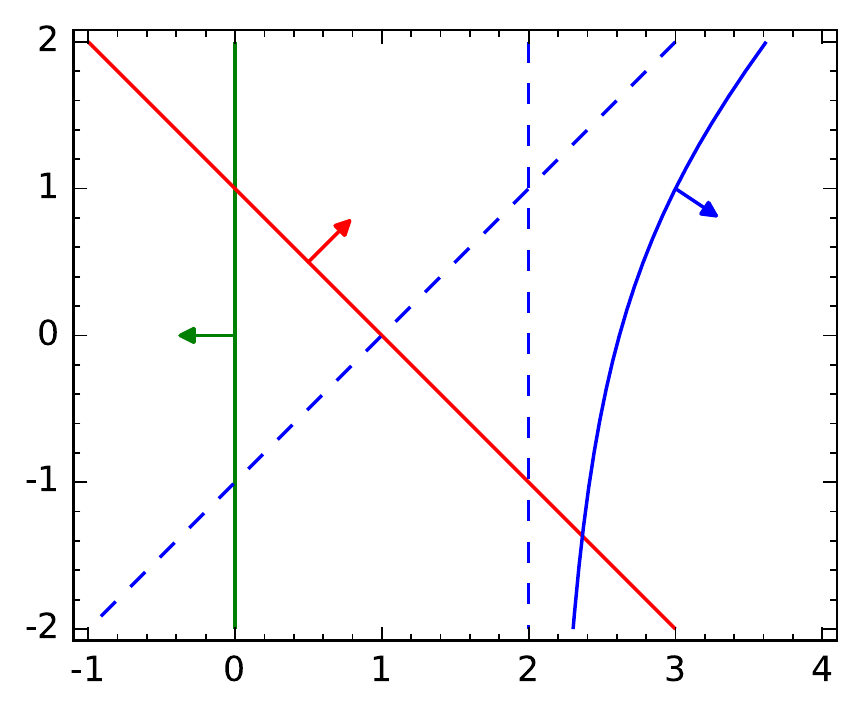}
    \end{center}
    \caption{Deformation of Example~\ref{ex:blocklinear}, $\varepsilon=1$. For $\varepsilon = 0$,
      the quadratic curve degenerates to Figure~\ref{fig:blocklinear}, as illustrated by the
      dashed blue lines.}
    \label{fig:deformed}
  \end{figure}

  The alternative spectrahedron $S(\Sigma)$ is given by the set of 
  symmetric block matrices 
  \[
    X = \diag \Big( \begin{bmatrix}X_{11}\end{bmatrix},
  \begin{bmatrix}X_{22}\end{bmatrix},\begin{bmatrix}X_{33} & X_{34} \\
  X_{34} & X_{44} \end{bmatrix} \Big)
  \]
  satisfying
  \begin{align}
    X_{11} = 1 - X_{44} + 2\, \varepsilon X_{34} & \ge 0, \label{eq:spectx11} \\
    X_{22} = X_{33} = \tfrac{1}{2} - X_{44} + \varepsilon X_{34} & \ge 0, \label{eq:spectx22} \\
    X_{44} & \ge 0, \label{eq:spectx44} \\
    (\tfrac{1}{2} - X_{44} + \varepsilon X_{34}) \cdot X_{44} - X_{34}^2 & \ge 
0. \label{eq:spect4}
  \end{align}
  In $(X_{44},X_{34})$-coordinates, $S(\Sigma)$ is the set bounded by the
  ellipse in Figure~\ref{fig:sdpextreme} (for $\varepsilon =
  1$). 
  For $\varepsilon = 0$, the ellipse becomes a circle.
  Independent of $\varepsilon$, i.e., for any $\varepsilon \ge 0$,
  there are two distinguished extreme
  points, namely $(X_{44},X_{34})=(0,0)$ and
  $(X_{44},X_{34})=(\tfrac{1}{2},0)$, corresponding to the matrices
  \begin{align*}
    \begin{pmatrix}
      1 &&&\\ &\tfrac{1}{2}&& \\ && \tfrac{1}{2} & 0 \\ && 0 & 0  
    \end{pmatrix}
    \quad\text{ and }\quad
    \begin{pmatrix}
      \tfrac{1}{2} &&&\\ & 0&& \\ && 0 & 0\\ && 0 & \tfrac{1}{2} 
    \end{pmatrix}
  \end{align*}
  in $S(\Sigma)$. The diagonals of these matrices are exactly the two vertices 
of the 
  alternative polyhedron as in Example~\ref{ex:blocklinear}. While the right 
  matrix corresponds to an IIS, the left matrix does not.

  \begin{figure}
    \centering
    \begin{subfigure}{.44\textwidth}
      \centering
      \includegraphics[width=\linewidth]{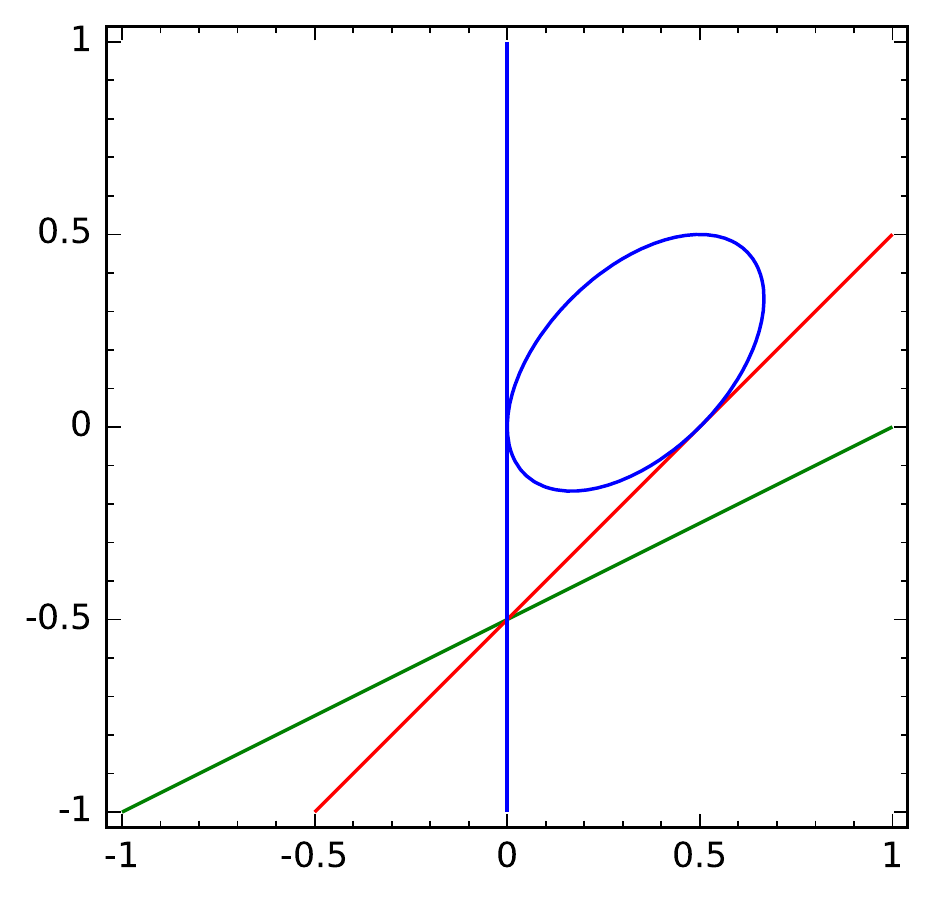}
      \caption{$\varepsilon=1$.}
    \end{subfigure} \hspace*{.1\textwidth}
    \begin{subfigure}{0.44\textwidth}
      \centering
      \includegraphics[width=\linewidth]{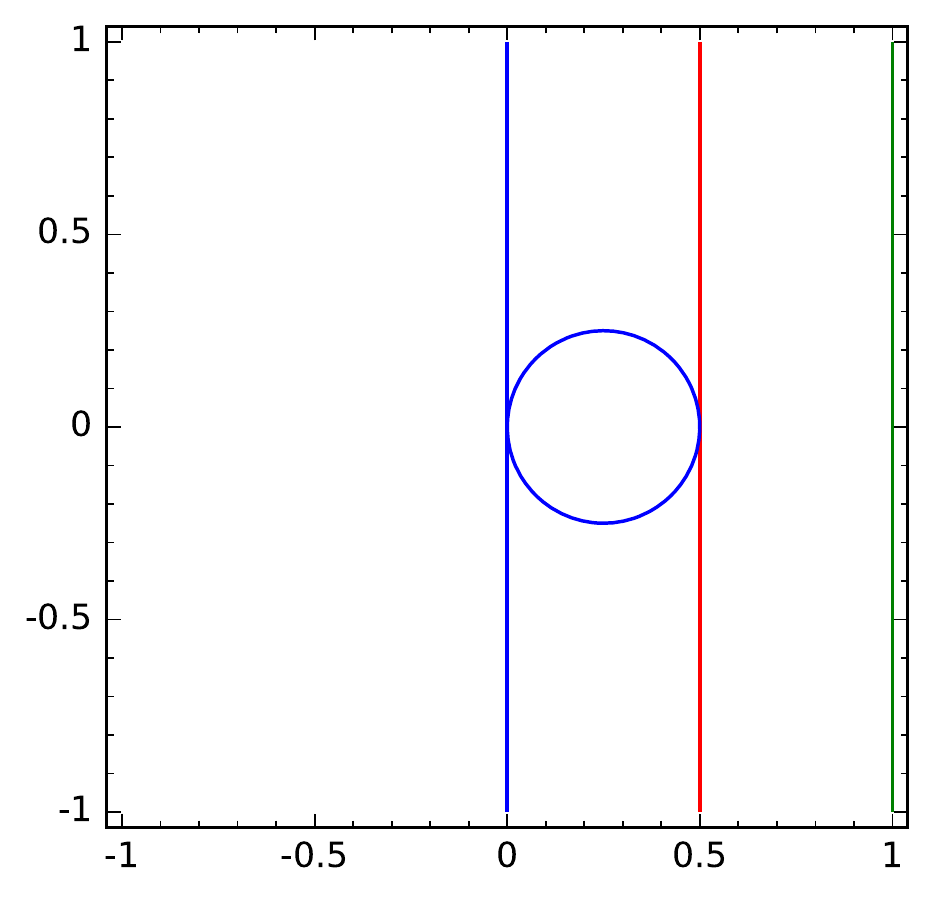}
      \caption{$\varepsilon=0$.}
    \end{subfigure}
    \caption{Alternative spectrahedron for Example~\ref{ex:Blocksdp} in 
      $(X_{44},X_{34})$-coordinates:
      $X_{44}$ is on the horizontal axis and $X_{34}$ on the vertical axis.
      The straight lines in green, red and blue correspond to 
      to~\eqref{eq:spectx11}, \eqref{eq:spectx22} and \eqref{eq:spectx44}, respectively.}
    \label{fig:sdpextreme}
  \end{figure}
\end{example}

These two examples motivate the question of how to compute IISs. By
Lemma~\ref{lemma:IIS_MinBlockSupport} it would suffice to compute a solution
with minimal block support. This can be obtained by a greedy approach in
which one iteratively solves semidefinite programs and fixes blocks to~0.
Note, however, that computing an IIS with minimal cardinality block support
is NP-hard already in the linear case, see~\cite{AmaPT03}.

In the particular case in which the alternative semidefinite system has a
unique solution, this algorithmic challenge simplifies to solving one
semidefinite program.
 In the next section we discuss universal conditions
under which the alternative semidefinite system has a unique solution.

\section{Universal Unique Solutions of Alternative Semidefinite Systems}
\label{se:recovery}

For a block matrix $X \in \sym^n$ with blocks $B_1, \ldots, B_k$, 
denote by $\sigma^B_+(X)$ the number of
blocks with at least one positive eigenvalue of $X$ and by $\sigma^B_-(X)$
the number of blocks with at least one negative eigenvalue.  Note that in
case of a positive semidefinite matrix $X$, the value $\sigma^B_+(X)$
coincides with $\card{\BS(X)}$.

The following statement is a generalization of
Theorem~1 in \cite{wang-tang-2009} to the case of block semidefinite systems.
Our proof employs a block semidefinite generalization of standard techniques from linear systems.
See also \cite{kdxh-2011} for a variant for linear systems and
Theorem~5 in \cite{wxt-2011} for a different (and non-block) generalization
of that theorem to the semidefinite case.

\begin{theorem}\label{thm:blocksparse1}
  For all psd block matrices $X^0 \in \sym^n$ with $\sigma^B_+(X^0) \leq t$, the set
  \[
  \{ X \succeq 0 \, : \, \sprod{A_i}{X} = \sprod{A_i}{X^0}, \, i \in [m]\}
  \]
  is a singleton if and only if for all symmetric $V \neq 0$, with
  $\sprod{A_i}{V} = 0$, $i \in [m]$, we have $\sigma^B_+(V) > t$ and
  $\sigma^B_-(V) > t$.
\end{theorem}

\begin{proof}
  Assume w.l.o.g.\ that there exists a symmetric $V \neq 0$ with
  $\sprod{A_i}{V} = 0$, $i \in [m]$, and $\sigma^B_-(V) \le t$. The proof
  of the case
  $\sigma^B_+(V) \le t$ is analogous, since the mapping $V \mapsto -V$
  exchanges positive and negative eigenvalues and we have
  $\sprod{A_i}{(-V)}=0$, $i \in [m]$, as well.
  For simplicity we further assume
  $\sigma^B_-(V)=t$.  Then, there exists a decomposition
  \[
  V \ = \ S\T D S,
  \]
  where $S$ is a regular block matrix (with respect to the blocks $B_1$,
  \dots, $B_k$) and $D = \diag(\lambda_1, \dots, \lambda_n)$ where
  $\lambda_i$ are the eigenvalues of $V$. In fact $S$ can be assumed to be
  orthonormal ($S\T = S^{-1}$) by performing a principal axis
  transformation for each block and combining the parts.

  By reordering we can assume that the negative eigenvalues appear
  in the first~$t$ blocks. We then define the diagonal matrices $D^1$, $D^2
  \in \R^{n \times n}$ with
  \[
  D^1_{ii} \define
  \begin{cases}
    - \lambda_i & \text{if } \lambda_i < 0,\\
    0 & \text{otherwise},
  \end{cases}
  \qquad
  D^2_{ii} \define
  \begin{cases}
    \lambda_i & \text{if } \lambda_i > 0,\\
    0 & \text{otherwise},
  \end{cases}
  \qquad
  i \in [n] \, .
  \]
  Then, $D^2 - D^1 = D$. We now obtain the block matrices
  \[
  X^1 = S\T D^1 S,\qquad X^2 = S\T D^2 S,
  \]
  with $X^1 \neq X^2$, $X^1 \succeq 0$, $X^2 \succeq 0$, and
  $\sigma^B_+(X^1) = t$. By construction $X^1_{B_i} = 0$ for all $i = t+1,
  \dots, k$. Moreover, for $i \in [m]$, we have
  \[
  \sprod{A_i}{X^1} = \sprod{A_i}{(S\T D^1 S)} = \sprod{A_i}{(S\T (D^2 - D) S)}.
  \]
  Since $\sprod{A_i}{(S\T D S)} = \sprod{A_i}{V} = 0$, this implies
  \[
  \sprod{A_i}{X^1} = \sprod{A_i}{(S\T D^2 S)} = \sprod{A_i}{X^2}.
  \]
  Hence, the set $\{X \succeq 0 \suchthat \sprod{A_i}{X}
  = \sprod{A_i}{X^1}, \, i \in [m]\}$ also contains $X^2$ and is thus not a
  singleton.

  Conversely, assume that there exists a psd matrix $X^0$ with
  $\sigma^B_+(X^0) \le t$ such that
  \[
  \{X \succeq 0 \suchthat \sprod{A_i}{X} = \sprod{A_i}{X^0}, \; i \in [m] \}
  \]
  is not a singleton. That is, there exists a matrix $\bar{X} \succeq 0$ with
  \[
  \sprod{A_i}{\bar{X}} = \sprod{A_i}{X^0}, \quad i \in [m]
  \]
  and $\bar{X} \neq X^0$. By the principal axis transformation, $X^0$ can
  be written as
  \[
  X^0 = S\T D^0 S
  \]
  with an orthonormal block matrix $S$ (w.r.t.\ the blocks $B_1, \dots, B_k$)
  and $D^0_{B_i} \geq 0$, $D^0_{B_i} \neq 0$ for $i = 1, \dots, t$ and
  $D^0_{B_i} = 0$ for $i = t + 1, \dots, k$. Setting $\bar{Y} = S
  \bar{X} S\T$, we have $\bar{Y} \succeq 0$ and $\bar{X} = S\T \bar{Y}
  S$.

  The block matrix $V = \bar{X} - X^0$ then satisfies $\sprod{A_i}{V} = 0$ and
  \[
  V \ = \ S\T (\bar{Y} - D^0) S.
  \]
  Then, in $\bar{Y} - D^0$ only the first $t$ blocks can have negative eigenvalues.
  Since the transformation matrix $S$ respects the block structure,
  we have $\sigma^B_-(V) \le t$.
\end{proof}

\begin{remark}
  In the special case in which all blocks have size 1,
  Theorem~\ref{thm:blocksparse1} can be stated as follows: In this case all
  matrices $A_i$, $i = 0, \dots, m$ are diagonal. Let $a^i = ((A_1)_{ii},
  \dots, (A_m)_{ii})$, and let $A$ be the matrix formed by the
  rows~$a^i$. Then, $\sprod{A_i}{X} = \sprod{A_0}{X^0}$ is equivalent to $A
  x = A x^0$. The condition states that for all $v \neq 0$ with $A v = 0$
  we have $\card{\{i \suchthat v_i < 0\}} > t$ and $\card{\{i \suchthat v_i
    > 0\}} > t$, which is Theorem~1 in \cite{wang-tang-2009}.
\end{remark}

\begin{example}\label{ex:BlocklinearUnique}
  Consider again the matrices $A_0, A_1, A_2$ from Example~\ref{ex:blocklinear}.
  Setting
  \[
  X^0 =
  \begin{pmatrix}
    \tfrac{1}{2} & 0 & 0 & 0 \\
    0 & 0 & 0 & 0 \\
    0 & 0 & 0 & 0 \\
    0 & 0 & 0 & \tfrac{1}{2}
  \end{pmatrix}
  \]
  yields $\sprod{A_0}{X^0} = -1$ and $\sprod{A_1}{X^0} = \sprod{A_2}{X^0} =
  0$. In this case $\sigma^B_+(X^0) = 2$. For the corresponding system of 
  equations to have $X^0$ as the unique solution, we would need $\sigma^B_+(V) > 2$ and
  $\sigma^B_-(V) > 2$ for all symmetric $V \neq 0$ with $\sprod{A_0}{V} =
  \sprod{A_1}{V} = \sprod{A_2}{V} = 0$. However,
  \[
  V =
  \begin{pmatrix}
    1 & 0 & 0 & 0 \\
    0 & 1 & 0 & 0 \\
    0 & 0 & 1 & 0 \\
    0 & 0 & 0 & -1
  \end{pmatrix}
  \]
  satisfies the equality constraints, but has $\sigma^B_-(V) = 1$, which is
  in accordance with Example~\ref{ex:blocklinear}, in which two extreme
  point solutions arise.
\end{example}

\begin{example}\label{ex:UniqueLPCase}
  Let $n$ be even and consider the $(n-1) \times n$ linear system of
  equations
  \begin{equation}
    \label{eq:systemeq2}
    Dv \define
    \begin{pmatrix*}
      1 &  1 &  0 & 0 & \cdots & 0 \\
      0 &  1 &  1 & 0 & \cdots & 0 \\
      \vdots && \ddots & \ddots & & \vdots \\
      0 & \cdots &  0 &  1 &  1 & 0 \\
      0 & \cdots &  0 &  0 &  1 & 1 \\
    \end{pmatrix*}
    v = 0.
  \end{equation}
  Then, form a symmetric matrix~$V = \diag(v)$. Let $A_i$, $i = 1, \dots, n-1$, be
  appropriate symmetric $n \times n$ matrices such that
  \[
  \sprod{A_i}{V} = 0,\quad i = 1, \dots, n-1,
  \]
  is equivalent to $Dv = 0$. Without any loss of generality,
  these are block matrices with respect to
  the blocks $B_1 = \{1\}$, \dots, $B_n = \{n\}$. In the notation of
  Theorem~\ref{thm:blocksparse1}, $m \define n-1$ and $V \neq 0$ with
  $\sprod{A_i}{V} = 0$ for all $i \in [m]$ is equivalent to $v \neq 0$ with
  $Dv = 0$.
  
  Then, $v_1 = -v_2 = v_3 = -v_4 = \dots = v_{n-1} = -v_n$. We can assume
  w.l.o.g.\ (by possible multiplication with $-1$) that $v_1 > 0$. Then,
  $v_1, v_3, \dots, v_{n-1}$ will be positive, while $v_2, v_4, \dots, v_n$
  will be negative. Thus, any solution~$V \neq 0$ to $\sprod{A_i}{V} = 0$,
  $i \in [m]$, satisfies $\sigma^B_+(V) = \sigma^B_-(V) = n/2$. By
  Theorem~\ref{thm:blocksparse1}, the system $\sprod{A_i}{X} =
  \sprod{A_i}{X^0}$, $i \in [m]$, has the unique (symmetric) solution $X^0
  \succeq 0$ if $\sigma^B_+(X^0) < n/2$. Note that the rank of the
  matrix~$D$ is $n-1$, which shows that the system has infinitely many
  solutions if $X^0$ is an arbitrary matrix.
\end{example}

\begin{remark}
  Consider the condition on~$V$ in Theorem~\ref{thm:blocksparse1}. The
  total number of blocks is at most~$n$, and if a block contributes both to
  $\sigma^B_+(V)$ and $\sigma^B_-(V)$, the block has to have at least
  size~2. Therefore, $\sigma^B_+(V) + \sigma^B_-(V) \leq n$. This implies
  that the largest $t$ for which $\sigma^B_+(V) > t$ and $\sigma^B_-(V) >
  t$ can hold is $\floor{n/2} - 1$. Example~\ref{ex:UniqueLPCase} shows
  that this bound is tight (if $n$ is odd, one can ignore a single variable
  in~$v$ and use the construction on the remaining part). Note that for even $n$
  this bound can only be attained in the LP-case, i.e., if all matrices are
  diagonal.
\end{remark}

\begin{example}\label{ex:UniqueSDPCase}
  Let $n$ be divisible by 3, define $k \define n/3$, and consider the $2 \times
  2$ blocks $B_1 = \{1,2\}$, \dots, $B_k = \{2k-1, 2k\}$. Take the same
  $(n-1) \times n$ linear system of equations as in
  Example~\ref{ex:UniqueLPCase} and fill in the variables of a solution~$v$
  into the symmetric $2k \times 2k$ block matrix $V$ as follows:
  \[
  V =
  \begin{pmatrix}
    v_1 & v_3 &    &     &        &   &   \\
    v_3 & v_2 &    &     &        &   &   \\
        &    & v_4 & v_6 &        &   &   \\
        &    & v_6 & v_5 &        &   &   \\
        &    &    &     & \ddots &   &    \\
        &    &    &    &        & v_{3k-2} & v_{3k} \\
        &    &    &    &        & v_{3k}  & v_{3k-1} \\
  \end{pmatrix}.
  \]
  Let $A_i$, $i = 1, \dots, m \define n-1$, be symmetric 
  $n \times n$ matrices such that
  \[
  \sprod{A_i}{V} = 0,\quad i = 1, \dots, n-1
  \]
  is equivalent to $Dv = 0$ from~\eqref{eq:systemeq2}. We can assume
  that the $A_i$ are block matrices for the
  above blocks. As in Example~\ref{ex:UniqueLPCase}, assuming that $v_1 >
  0$, the equations imply that $v_1 = v_3 = \dots = v_{n-1} > 0$, while
  $v_2 = v_4 = \dots = v_n < 0$. Thus, denoting $\lambda = v_1$, each 
  $2 \times 2$ block has the following structure:
  \[
  \begin{pmatrix*}[r]
    \lambda & \lambda \\
    \lambda & -\lambda
  \end{pmatrix*}
  \quad\text{or}\quad
  \begin{pmatrix*}[r]
    -\lambda & -\lambda \\
    -\lambda & \lambda
  \end{pmatrix*}.
  \]
  In both cases, the eigenvalues are $\pm \sqrt{2}\, \lambda$. Therefore,
  each block is counted both in $\sigma^B_+(V)$ and in $\sigma^B_-(V)$.
  Thus, any solution~$V \neq 0$ to $\sprod{A_i}{V} = 0$, $i \in [m]$, will
  satisfy $\sigma^B_+(V) = \sigma^B_-(V) = k = n/3$. By
  Theorem~\ref{thm:blocksparse1}, the system $\sprod{A_i}{X} =
  \sprod{A_i}{X^0}$, $i \in [m]$, has the unique (symmetric) solution $X^0
  \succeq 0$ if $\sigma^B_+(X^0) < n/3$.
\end{example}

\begin{remark}
  Note that the uniqueness conditions of Example~\ref{ex:UniqueSDPCase}
  include matrices~$X^0$ with negative entries, which would not be allowed
  in the LP-case (as in Example~\ref{ex:UniqueLPCase}); for example, if
  $k-1$ blocks of $X^0$ consist of the positive definite matrices
  \[
  \begin{pmatrix*}[r]
    2   & -1 \\
    -1  &  2 
  \end{pmatrix*}
  \]
  and 0-blocks otherwise. This shows that while the size of~$t$ is possibly
  smaller than in the LP-case, the general spectrahedron case allows for a
  wider range of cases of $X^0$ in which uniqueness appears.
\end{remark}

\section{Perspectives and Open Problems}
\label{se:perspectives}

In Section~\ref{se:recovery}, we
have provided a criterion for particular semidefinite block systems to have
a unique feasible solution. If this particular situation does not arise, it
is an open question whether one can obtain an IIS by solving a single
semidefinite program.

By Lemma~\ref{lemma:IIS_MinBlockSupport} it would suffice to find solutions
of minimal block support. For a matrix~$X$ the number of nonzero blocks can
be written as
\[
\norm{X}_{2,0} \define \big\lVert (\norm{X_{B_1}}_2, \dots, \norm{X_{B_k}}_2) \big\rVert_0,
\]
where $\norm{x}_0$ denotes the number of nonzeros in a vector $x$. Thus,
it would suffice to solve the following problem to find an IIS:
\begin{equation}
  \label{eq:0norm}
  \min \big\{ \norm{X}_{2,0} \suchthat X \in S(\Sigma) \big\}.
\end{equation}
Unfortunately, the $\norm{\cdot}_{2,0}$ ``norm'' is nonconvex and thus hard
to handle, for instance, \eqref{eq:0norm} is NP-hard. However, for linear
systems recent developments, see, e.g., \cite{EldKB10,elhamifar-vidal-2012,Lin2013},
suggest to replace $\norm{X}_{2,0}$ by
\[
\norm{X}_{2,1} \define \big\lVert (\norm{X_{B_1}}_2, \dots,
\norm{X_{B_k}}_2) \big\rVert_1 = \sum_{i=1}^k \norm{X_{B_i}}_2,
\]
which leads to the following convex optimization problem:
\begin{equation}
  \label{eq:2norm}
  \min \Big\{ \sum_{i=1}^k \norm{X_{B_i}}_2 \suchthat X \in S(\Sigma) \Big\}.
\end{equation}

\begin{lemma}
  Problem~\eqref{eq:2norm} can be formulated as SDP.
\end{lemma}

\begin{proof}
  Use the second order-cone condition
  \[
  \Big\{ (x,t) \suchthat \big( \sum x_i^2\big)^{1/2} \le t \Big\}
  \]
  to represent $\norm{X_{B_i}}_2 \le t_i$ with new variables $t_i$ and
  minimize the objective function $\sum_{i=1}^k t_i$. It is well-known
  that second order conditions are special cases of semidefinite conditions
  (see, e.g., \cite{tuncel2010}).  Since $X \in S(\Sigma)$ is already a
  positive semidefinite condition, this concludes the proof.
\end{proof}

An interesting line of future research would investigate conditions under
which~\eqref{eq:2norm} provides an optimal solution for~\eqref{eq:0norm},
which would try to generalize the above mentioned results from the linear
to the block semidefinite case.

As pointed out by an anonymous referee, it would also be interesting to
understand properties, under which an infeasible semidefinite system
satisfies the converse of Theorem~\ref{th:iis-gives-extreme}, as well as
properties, under which every extreme point of the alternative spectrahedron
has minimal block support.

Finally, it remains as a natural question to study the combination
of our methods with exact duality theory versions for semidefinite
programming, such as the reformulation technique of \cite{liu-pataki-2015}.

\section{Conclusions}

We have shown that one direction of the Gleeson-Ryan-Theorem for infeasible
linear systems generalizes to infeasible block semidefinite systems, but the
other direction does not. To overcome the situation to identify IISs, we
have given a unique recovery characterization. Both the algorithmic question
touched in Section~\ref{se:perspectives} and the practical question of
how to effectively exploit IISs of semidefinite systems within semidefinite
integer programming solvers deserve further study.

\medskip

\noindent
{\bf Acknowledgment.} We thank the anonymous referees for helpful
suggestions.


\bibliographystyle{plain}

\end{document}